\documentclass[reqno,12pt]{amsart}
\usepackage[left=1.2in, right = 1.2in, top=1in]{geometry}
\usepackage{amsmath, amssymb, amsthm, mathrsfs, color, times, textcomp, verbatim}
\usepackage{amstext}
\usepackage{appendix}
\usepackage{verbatim}
\usepackage{cases}
\usepackage{url}
\usepackage{graphicx}
\usepackage{tikz}
\newtheorem{theorem}{Theorem}[section]

\newtheorem{lemma}{Lemma}[section]

\newtheorem{claim}{claim}

\allowdisplaybreaks[4]
\usepackage{fancyhdr}
\usepackage{hyperref}
\title{A remark for characterizing blowup introduced by Giga and Kohn}

\author{Wangzhe Wu}

\address{Institute of Mathematics \\Academy of Mathematics and Systems Science, Chinese
Academy of Sciences, Beijing, 100190, China}
\email{wuwz18@mail.ustc.edu.cn}

\begin{document}

\begin{abstract}
	Giga and Kohn studied the blowup solutions for the equation $v_{t} - \Delta v - |v|^{p - 1} v = 0 $ and characterized the asymptotic 
behavior of $v$ near a singularity. In the proof, they reduced the problem to a Liouville theorem for the equation $\Delta u - \frac{1}{2} x \cdot \nabla u + |u|^{p - 1} u - \beta u = 0$ where $\beta = \frac{1}{p - 1}$ and $|u|$ is bounded. This article is a remark for their work and we will show when $u \geq 0$, the boundedness condition for $|u|$ can be removed. 
	\end{abstract}

	\pagestyle{fancy}

\fancyhead{}
\fancyhead[CO]{A remark for characterizing blowup introduced by Giga and Kohn}

	\maketitle

	\setlength{\headheight}{15pt}

	\section{introduction}
	
	Define 
	\begin{equation}
		p_* = \left\{ \begin{array}{l}
			\infty , \quad\text{if } 1 \leq n \leq 2\\
			\frac{n + 2}{n - 2}, \quad\text{if } n \geq 3 .
		\end{array}
		\right .
	\end{equation}
	Consider the equation:
	
	\begin{equation}\label{0929ori1}
		\begin{aligned}
			&v_{t} - \Delta v - |v|^{p - 1} v = 0 \text{ in $D \times (0, T)$},\\
		&v = 0	\text{ on $\partial D$},
		\end{aligned}
	\end{equation}
	where $1 < p < p_*$ and $n \geq 1$ . And $T$ is the maximal existence time of the $L^\infty$ solution $v$. Then $v$ will blow up when $t$ tends to $T$. Giga and Kohn in \cite{MR876989} got that when $D$ is a bounded convex domain or $\mathbb R^n$ and $u \geq 0$, then 
	\begin{equation}\label{0929con1}
		\sup_{D \times (0, T)} v(x, t) \cdot (T - t)^{\frac{1}{p - 1}} < \infty.
	\end{equation}
	In the proof, their fundamental tool is the change of both dependent and independent variables defined by
	\begin{equation}
		w_{a}(y, s) := (T - t)^{\beta} v(a + y \sqrt{T - t}, t), \quad s = -\log (T - t).
	\end{equation}
	Then one computes that $w = w_{a}$ solves a new parabolic equation in $(y, s)$ from \eqref{0929ori1}:
	\begin{equation}
		w_s - \Delta w + \frac{1}{2}y\cdot \nabla w + \beta w - |w|^{p - 1}w = 0,
	\end{equation}
	and the blowup time $T$ corresponds to $s = \infty$.
	Furthermore, if the blowup rate \eqref{0929con1} holds (under this assumption, the condition $u \geq 0$ is not needed), in \cite{MR784476} they got a finer description of $u$ near blowup:
	\begin{equation}
		\lim_{s \rightarrow \infty} w(y, s) = 0  \text{ or $\pm k$},
	\end{equation}
	uniformly for $|y| < C$ and $k = (p - 1)^{-\frac{1}{p - 1}}$. In this proof, they found the monotonicity of the energy functional
	\begin{equation}
		E(w) = \int_{\mathbb R^n} \left[ \frac{1}{2}|\nabla w|^2 + \frac{1}{2(p - 1)}w^2 - \frac{1}{p + 1}|w|^{p + 1}\right] e^{-\frac{1}{4}|y|^2} dy.
	\end{equation}
	Then they deduced that when $s$ tends to $\infty$, $w_s$ will tend to zero. Furthermore,  Kelei Wang, Juncheng Wei and Ke Wu \cite{Keleiwang} proved that when $p > \frac{n + 2}{n - 2}$ and $n \geq 3$, not only the constant solution has the lowest energy $E(w)$, but there is a gap to the second lowest. Anyway when $u$ is stationary, the problem in \cite{MR784476} is reduced to one Liouville theorem for the following  equation under the assumption that $|u|$ is bounded:
	\begin{equation}\label{8.9equ1}
		\Delta u - \frac{1}{2} x \cdot \nabla u + |u|^{p - 1}u - \beta u = 0,  \quad x \in \mathbb R^n,
	\end{equation}
	with $\beta = \frac{1}{p - 1}$. 
	\begin{theorem}[Proposition 2 in \cite{ MR784476} ]\label{0408thm}
		Suppose $1 < p \leq p_*$ and $n\geq 1$ . If $u$ is a classical solution of \eqref{8.9equ1} and $|u|$ is bounded, then $|\nabla u| \equiv 0$.
	\end{theorem}
	The proof of Theorem \ref{0408thm} is based a Pohozaev's identity: multiply \eqref{8.9equ1} by $\nabla u \cdot x$ and use integration by parts. This method is widely used to get Liouville theorems. Using the Pohozaev's identity, Souplet \cite{SOUPLET20091409} solved the Lane-Emden conjecture for dimension $n \leq 4$. Besides under the assumption
	\begin{equation}\label{0409equ1}
		\int_{\mathbb R^n} |\nabla u|^2 + u^{p + 1} < \infty,
	\end{equation}
	  Byeon, Ikoma, Malchiodi and Mari \cite{byeon2024compactnessmonotonicitynonsmoothcritical} used the Pohozaev's identity to get the Liouville theorem for the equation 
	\begin{equation}
		div\left( \frac{\nabla u}{\sqrt{1 - |\nabla u|^2}} \right) + u^{p - 1}u = 0.
	\end{equation}
	However this method always needs a global estimate (such as \eqref{0409equ1} ). This is also the reason why Souplet only dealt with the case $n \leq 4$. Similarly in \cite{ MR784476}, Giga and Kohn used the boundedness of $|u|$ to get the desired global estimate.	
	 This article is a remark for Theorem \ref{0408thm} and we will show that the boundedness condition for $|u|$ can be replaced by the condition that $u\geq 0$.

	\begin{theorem}\label{0928thm}
		Suppose $1 < p < p_*$ and $n\geq 1$ . If $u \geq 0$ is a classical solution of \eqref{8.9equ1}, then $|\nabla u| \equiv 0$.
	\end{theorem}
	In this article, we will introduce a new idea to get the global estimate for the Pohozaev's identity, which is the key point of Theorem \ref{0928thm} and will be explained in Lemma \ref{Thm1}. The idea is based a differential identity and integration by parts. We are greatly inspired by the works of Gidas and Spruck \cite{MR615628}: they developed the method and proved the Liouville theorem for the equation
	\begin{equation}\label{GS}
		\Delta u + u^p = 0, u \geq 0, 1 < p < p_*.
	\end{equation}
	Quittner \cite{MR4255053} got the Liouville theorem for  
	\begin{equation}\label{Quittner}
		v_{t} - \Delta v - v^p = 0, v > 0, 1 < p < p_*\text{ in $\mathbb R \times \mathbb R^n$.}
	\end{equation}
	And the main idea of Quittner's proof is similar to Giga-Kohn's: he used refined energy estimates for suitably rescaled solutions which yield a positive stationary solution. So the problem \eqref{Quittner} is reduced to \eqref{GS}.
	For more elliptic equations, M. F. Bidaut-V\'eron, M. Garcia-Huidobro and L. V\'eron \cite{MR3959864, MR4150912,  MR4043662} used integration by parts to get the Liouville theorem for the equations:
	\begin{equation}
		\Delta u + u^p |\nabla v|^q = 0, u \geq 0,
	\end{equation}
	and
	\begin{equation}
		\Delta u + N u^p + M|\nabla v|^q = 0, M > 0, N > 0, u \geq 0.
	\end{equation}
	For more applications, one can refer to \cite{MR4753063, MR4519639, ma2023liouvilletheoremellipticequations, Ma2024liouvilletheoremellipticequations,wu2023liouvilletheoremquasilinearelliptic}.

	\textbf{Acknowledgement } The author was supported by National Natural Science Foundation of China (Grants 11721101 and 12141105) and National Key Research and Development Project (Grants SQ2020YFA070080). The author would like to thank Prof. Xi-Nan Ma for valuable comments and suggestions.
	
	\section{Proof of theorem}
	In this section, we will use ``$\lesssim$'' to drop some positive constants independent of $R$ and the value of $u$ in $\mathbb R^n \backslash B_{r_0}$, where $r_0 > 1$ is a fixed constant. Besides all of the integrations are over $\mathbb R^n$.
	Before we prove Theorem \ref{0928thm}, the following lemma is necessary:
	\begin{lemma}\label{Thm1}
		Suppose $1 < p < p_*$ and  $n \geq 2$. Consider the equation:
		\begin{equation}\label{0910equ1}
		\Delta u - \lambda x \cdot \nabla u + u^{p} - \beta u = 0, \quad u \geq 0,\quad x \in \mathbb R^n.
	\end{equation}
		For any  $\lambda > 0, \beta \geq 0$ , if $u \geq 0$ is a solution of \eqref{0910equ1} , then for any $\gamma > 0$, 
		\begin{equation}
			\int_{\mathbb R^n}  u^{p + 1} e^{-\gamma |x|^2} + |\nabla u|^2 e^{-\gamma|x|^2} < \infty .
		\end{equation}
		
	\end{lemma}

	\begin{proof}[Proof of Lemma \ref{Thm1}]
		
	By strong maximum principle, we can suppose $u > 0$.  Define for any $1 \leq i, j \leq n$,
	\begin{align}
		E_{i j} &= u_{i j} - \frac{1}{n}\Delta u \delta_{ij},\\
		L_{i j} &= u_i u_j - \frac{1}{n}|\nabla u|^2 \delta_{i j}.
	\end{align}
	In the following contents, we always use the Einstein summation and the index $i, j$ are summed from $1$ to $n$.
	Firstly, we know
	\begin{align*}
		&u^{\alpha}e^{-\gamma |x|^2}(\Delta u)^2 \\
		=& \left( u^{\alpha}e^{-\gamma |x|^2}\Delta u u_{i} \right)_i - \alpha u^{\alpha - 1}|\nabla u|^2 e^{-\gamma |x|^2}\Delta u + 2\gamma u^{\alpha}e^{-\gamma |x|^2}\Delta u u_{i} x_i - u^{\alpha}e^{-\gamma |x|^2}\Delta u_i u_{i} \\
		=& \left( u^{\alpha}e^{-\gamma |x|^2}\Delta u u_{i} \right)_i - \alpha u^{\alpha - 1}|\nabla u|^2 e^{-\gamma |x|^2}\Delta u + 2\gamma u^{\alpha}e^{-\gamma |x|^2}\Delta u u_{i} x_i - \left( u^{\alpha}e^{-\gamma |x|^2} u_{i j} u_{i} \right)_{j} \\
		&+ \alpha  u^{\alpha - 1}e^{-\gamma |x|^2} u_{i j} u_{i}u_{j} - 2\gamma  u^{\alpha}e^{-\gamma |x|^2} u_{i j} u_{i} x_{j} +  u^{\alpha}e^{-\gamma |x|^2} u_{i j}^2\\
		=& \left( u^{\alpha}e^{-\gamma |x|^2}\Delta u u_{i} \right)_i - \alpha u^{\alpha - 1}|\nabla u|^2 e^{-\gamma |x|^2}\Delta u + 2\gamma u^{\alpha}e^{-\gamma |x|^2}\Delta u u_{i} x_i - \left( u^{\alpha}e^{-\gamma |x|^2} u_{i j} u_{i} \right)_{j} \\
		&+ \alpha  u^{\alpha - 1}e^{-\gamma |x|^2} \left( E_{i j} + \frac{1}{n}\Delta u \delta_{i j}\right) u_{i}u_{j} - 2\gamma  u^{\alpha}e^{-\gamma |x|^2} \left( E_{i j} + \frac{1}{n}\Delta u \delta_{i j}\right) u_{i} x_{j} \\
		&+ u^{\alpha}e^{-\gamma |x|^2} E_{i j}^2 + \frac{1}{n}u^{\alpha}e^{-\gamma |x|^2}(\Delta u)^2,
	\end{align*}
	
	\begin{equation}\label{0830equ1}
		\begin{aligned}
			\Rightarrow & \frac{n - 1}{n} u^{\alpha}e^{-\gamma |x|^2}(\Delta u)^2 \\
		=& \left( u^{\alpha}e^{-\gamma |x|^2}\Delta u u_{i} \right)_i - \left( u^{\alpha}e^{-\gamma |x|^2} u_{i j} u_{i} \right)_{j} - \frac{n - 1}{n} \alpha u^{\alpha - 1}|\nabla u|^2 e^{-\gamma |x|^2}\Delta u \\
		&+ 2\gamma \frac{n - 1}{n} u^{\alpha}e^{-\gamma |x|^2}\Delta u u_{i} x_i   + \alpha  u^{\alpha - 1}e^{-\gamma |x|^2} E_{i j}  u_{i}u_{j} - 2\gamma  u^{\alpha}e^{-\gamma |x|^2}E_{i j}  u_{i} x_{j}  + u^{\alpha}e^{-\gamma |x|^2} E_{i j}^2 .
		\end{aligned}
	\end{equation}
	Next, using the equation \eqref{0910equ1} we also have
	\begin{align*}
		&u^{\alpha}e^{-\gamma |x|^2}(\Delta u)^2 \\
		=& \lambda u^{\alpha}e^{-\gamma |x|^2}\Delta u x_i u_i  + \beta u^{\alpha + 1}e^{-\gamma |x|^2}\Delta u - u^{\alpha + p}e^{-\gamma |x|^2}\Delta u\\
		=& \lambda u^{\alpha}e^{-\gamma |x|^2}\Delta u x_i u_i + \left(\beta u^{\alpha + 1}e^{-\gamma |x|^2}u_i \right)_i - (\alpha + 1)\beta u^{\alpha }|\nabla u|^2 e^{-\gamma |x|^2} \\
		&+ 2\gamma \beta u^{\alpha + 1}e^{-\gamma |x|^2}u_i x_i - \left( u^{\alpha + p}e^{-\gamma |x|^2}u_{i} \right)_i + (\alpha + p)u^{\alpha + p - 1}|\nabla u|^2 e^{-\gamma |x|^2} - 2\gamma u^{\alpha + p}e^{-\gamma |x|^2}u_{i} x_i\\
		=&\lambda u^{\alpha}e^{-\gamma |x|^2}\Delta u x_i u_i + \left(\beta u^{\alpha + 1}e^{-\gamma |x|^2}u_i \right)_i - (\alpha + 1)\beta u^{\alpha }|\nabla u|^2 e^{-\gamma |x|^2} + 2\gamma \beta u^{\alpha + 1}e^{-\gamma |x|^2}u_i x_i \\
		&- \left( u^{\alpha + p}e^{-\gamma |x|^2}u_{i} \right)_i - 2\gamma u^{\alpha}e^{-\gamma |x|^2}u_{i} x_i \left( -\Delta u + \lambda u_j x_j +\beta u \right) \\
		&+ (\alpha + p)u^{\alpha - 1}|\nabla u|^2 e^{-\gamma |x|^2}\left( -\Delta u + \lambda u_i x_i +\beta u \right) ,
	\end{align*}
	\begin{equation}\label{0830equ2}
		\begin{aligned}
			\Rightarrow &u^{\alpha}e^{-\gamma |x|^2}(\Delta u)^2 = \left( u^{\alpha}e^{-\gamma |x|^2}\Delta u u_{i} \right)_i - \left( \frac{1}{2} u^{\alpha }e^{-\gamma |x|^2}u_j x_j u_i \right)_i + \left( \lambda + 2\gamma\right) u^{\alpha}e^{-\gamma |x|^2}\Delta u x_i u_i \\
			& + (p - 1)\beta u^{\alpha }|\nabla u|^2 e^{-\gamma |x|^2} - 2 \gamma \lambda u^{\alpha}e^{-\gamma |x|^2}u_{i} x_i  u_j x_j + (\alpha + p)u^{\alpha - 1}|\nabla u|^2 e^{-\gamma |x|^2}\left( -\Delta u + \lambda u_i x_i  \right) .
		\end{aligned}
	\end{equation}
	
	Combine \eqref{0830equ1} and \eqref{0830equ2} to eliminate the term $u^{\alpha}e^{-\gamma |x|^2}(\Delta u)^2$,
	
	\begin{align}
		0&= \left( u^{\alpha}e^{-\gamma |x|^2}\Delta u u_{i} \right)_i - \left( u^{\alpha}e^{-\gamma |x|^2} u_{i j} u_{i} \right)_{j} - \frac{n - 1}{n} \alpha u^{\alpha - 1}|\nabla u|^2 e^{-\gamma |x|^2}\Delta u \notag \\
		&+ 2\gamma\frac{n - 1}{n}  u^{\alpha}e^{-\gamma |x|^2}\Delta u u_{i} x_i + \alpha  u^{\alpha - 1}e^{-\gamma |x|^2} E_{i j}  u_{i}u_{j} - 2\gamma  u^{\alpha}e^{-\gamma |x|^2}E_{i j}  u_{i} x_{j}  + u^{\alpha}e^{-\gamma |x|^2} E_{i j}^2 \notag\\
		&- \frac{n - 1}{n}\Big[  - \left( u^{\alpha + p}e^{-\gamma |x|^2}u_{i} \right)_i + \left(\beta u^{\alpha + 1}e^{-\gamma |x|^2}u_i \right)_i + \left( \lambda + 2\gamma\right) u^{\alpha}e^{-\gamma |x|^2}\Delta u x_i u_i \notag\\
		& + (p - 1)\beta u^{\alpha }|\nabla u|^2 e^{-\gamma |x|^2} - 2\gamma \lambda u^{\alpha}e^{-\gamma |x|^2}u_{i} x_i  u_j x_j + (\alpha + p)u^{\alpha - 1}|\nabla u|^2 e^{-\gamma |x|^2}\left( -\Delta u + \lambda u_i x_i  \right) \Big]  \notag\\
		&= W_1  + \frac{n - 1}{n} p u^{\alpha - 1}|\nabla u|^2 e^{-\gamma |x|^2}\Delta u - \frac{n - 1}{n}\lambda  u^{\alpha}e^{-\gamma |x|^2}\Delta u u_{i} x_i  \notag\\
		&+ \alpha  u^{\alpha - 1}e^{-\gamma |x|^2} E_{i j}  u_{i}u_{j} - 2\gamma  u^{\alpha}e^{-\gamma |x|^2}E_{i j}  u_{i} x_{j}  + u^{\alpha}e^{-\gamma |x|^2} E_{i j}^2 \notag\\
		\label{0830equ3} &- \frac{n - 1}{n}\Big[    (p - 1)\beta u^{\alpha }|\nabla u|^2 e^{-\gamma |x|^2} - 2\gamma \lambda u^{\alpha}e^{-\gamma |x|^2}u_{i} x_i  u_j x_j +  \lambda (\alpha + p)u^{\alpha - 1}|\nabla u|^2 e^{-\gamma |x|^2} u_i x_i \Big],
	\end{align}
	where
	\begin{align*}
		W_1 &= \left( u^{\alpha}e^{-\gamma |x|^2}\Delta u u_{i} \right)_i - \left( u^{\alpha}e^{-\gamma |x|^2} u_{i j} u_{i} \right)_{j} +\frac{n - 1}{n} \left( u^{\alpha + p}e^{-\gamma |x|^2}u_{i} \right)_i - \frac{n - 1}{n}\left(\beta u^{\alpha + 1}e^{-\gamma |x|^2}u_i \right)_i\\
		 &= \frac{1}{n}\left( u^{\alpha}e^{-\gamma |x|^2}\Delta u u_{i} \right)_i - \left( u^{\alpha}e^{-\gamma |x|^2} u_{i j} u_{i} \right)_{j} +\frac{n - 1}{n} \left( u^{\alpha }e^{-\gamma |x|^2}u_j x_j u_{i} \right)_i .
	\end{align*}
	On the other hand, since 
	\begin{align*}
		&u^{\alpha - 1}|\nabla u|^2 e^{-\gamma |x|^2}\Delta u\\
		=&(u^{\alpha - 1}|\nabla u|^2 e^{-\gamma |x|^2}u_i)_i - (\alpha - 1) u^{\alpha - 2}|\nabla u|^4 e^{-\gamma |x|^2} - 2u^{\alpha - 1} e^{-\gamma |x|^2}u_{i j}u_i u_j + 2\gamma u^{\alpha - 1}|\nabla u|^2 e^{-\gamma |x|^2}u_i x_i,
	\end{align*}
	\begin{equation}\label{0908equ1}
		\begin{aligned}
			\Rightarrow & \frac{n + 2}{n} u^{\alpha - 1}|\nabla u|^2 e^{-\gamma |x|^2}\Delta u\\
		=&(u^{\alpha - 1}|\nabla u|^2 e^{-\gamma |x|^2}u_i)_i - (\alpha - 1) u^{\alpha - 2}|\nabla u|^4 e^{-\gamma |x|^2} - 2u^{\alpha - 1} e^{-\gamma |x|^2}E_{i j}u_i u_j + 2\gamma u^{\alpha - 1}|\nabla u|^2 e^{-\gamma |x|^2}u_i x_i,
		\end{aligned}
	\end{equation}
	and 
	\begin{align*}
		&u^{\alpha}e^{-\gamma |x|^2}\Delta u u_{i} x_i =\left( u^{\alpha}e^{-\gamma |x|^2}u_j u_{i} x_i \right)_j - \alpha u^{\alpha - 1}|\nabla u|^2 e^{-\gamma |x|^2} u_{i} x_i + 2\gamma u^{\alpha}e^{-\gamma |x|^2}u_j x_j u_{i} x_i\\
		& - u^{\alpha}e^{-\gamma |x|^2}u_j u_{ij} x_i - u^{\alpha}|\nabla u|^2 e^{-\gamma |x|^2}, 
	\end{align*}
	\begin{equation}\label{0908equ2}
		\begin{aligned}
			\Rightarrow & \frac{n + 1}{n} u^{\alpha}e^{-\gamma |x|^2}\Delta u u_{i} x_i =\left( u^{\alpha}e^{-\gamma |x|^2}u_j u_{i} x_i \right)_j - \alpha u^{\alpha - 1}|\nabla u|^2 e^{-\gamma |x|^2} u_{i} x_i + 2\gamma u^{\alpha}e^{-\gamma |x|^2}u_j x_j u_{i} x_i \\
			&- u^{\alpha}e^{-\gamma |x|^2} E_{ij} x_i u_j - u^{\alpha}|\nabla u|^2 e^{-\gamma |x|^2},
		\end{aligned}
	\end{equation}
	then \eqref{0830equ3} becomes
	\begin{align*}
		0&= W_2  + \frac{n - 1}{n + 2} p \Big[ - (\alpha - 1) u^{\alpha - 2}|\nabla u|^4 e^{-\gamma |x|^2} - 2u^{\alpha - 1} e^{-\gamma |x|^2}E_{i j}u_i u_j + 2\gamma u^{\alpha - 1}|\nabla u|^2 e^{-\gamma |x|^2}u_i x_i \Big]\\
		&  -\frac{n - 1}{n + 1}\lambda  \Big( - \alpha u^{\alpha - 1}|\nabla u|^2 e^{-\gamma |x|^2} u_{i} x_i + 2\gamma u^{\alpha}e^{-\gamma |x|^2}u_j x_j u_{i} x_i - u^{\alpha}e^{-\gamma |x|^2} E_{ij} x_i u_j\\
		&- u^{\alpha}|\nabla u|^2 e^{-\gamma |x|^2} \Big) + \alpha  u^{\alpha - 1}e^{-\gamma |x|^2} E_{i j}  u_{i}u_{j}  - 2\gamma  u^{\alpha}e^{-\gamma |x|^2}E_{i j}  u_{i} x_{j}  + u^{\alpha}e^{-\gamma |x|^2} E_{i j}^2\\
		&+ 2\gamma \lambda \frac{n - 1}{n}  u^{\alpha}e^{-\gamma |x|^2}u_{i} x_i  u_j x_j   - \frac{n - 1}{n}\lambda (\alpha + p) u^{\alpha - 1}|\nabla u|^2 e^{-\gamma|x|^2}x_i u_i  - \frac{n - 1}{n} (p - 1)\beta  u^{\alpha }|\nabla u|^2 e^{-\gamma |x|^2},
	\end{align*}
	where
	\begin{align*}
		W_2 &= W_1 + \frac{n - 1}{n + 2} p(u^{\alpha - 1}|\nabla u|^2 e^{-\gamma |x|^2}u_i)_i  -\frac{n - 1}{n + 1}\lambda \left( u^{\alpha}e^{-\gamma |x|^2}u_j u_{i} x_i \right)_j\\
		&= D_1 \left( u^{\alpha}e^{-\gamma |x|^2}\Delta u u_{i} \right)_i + D_2\left( u^{\alpha}e^{-\gamma |x|^2} u_{i j} u_{i} \right)_{j} + D_3 (u^{\alpha - 1}|\nabla u|^2 e^{-\gamma |x|^2}u_i)_i \\
		&+ D_4 \left( u^{\alpha}e^{-\gamma |x|^2}u_j u_{i} x_i \right)_j .
	\end{align*}
	Using the fact that $|\nabla u|^4 = \frac{n}{n - 1}L_{i j}^2$,  we finally get that
	\begin{align}\label{0408equ2}
		 0&= W_2 + \underset{I_1}{ u^{\alpha}e^{-\gamma |x|^2} E_{i j}^2 } -  \underset{I_2}{\frac{n}{n + 2} p  (\alpha - 1) u^{\alpha - 2}L_{i j}^2 e^{-\gamma |x|^2}} + \underset{I_3}{\left( \alpha - 2p\frac{n - 1}{n + 2}  \right) u^{\alpha - 1}e^{-\gamma |x|^2} E_{i j} L_{i j}} \notag\\
		 &+ \left( - 2\gamma  +\frac{n - 1}{n + 1}\lambda\right)  u^{\alpha}e^{-\gamma |x|^2}E_{i j}  u_{i} x_{j} + \left( 2\gamma \lambda\frac{n - 1}{n}   - 2\gamma \lambda  \frac{n - 1}{n + 1}  \right) u^{\alpha}e^{-\gamma |x|^2}u_{i} x_i  u_j x_j  \notag\\
		&+ \Bigg[  - \frac{n - 1}{n} \lambda(\alpha + p) + 2\gamma\frac{n - 1}{n + 2} p + \frac{n - 1}{n + 1}\lambda \alpha\Bigg]u^{\alpha - 1}|\nabla u|^2 e^{-\gamma|x|^2}x_i u_i \notag\\
		& + \Bigg[ - \frac{n - 1}{n} (p - 1)\beta + \frac{n - 1}{n + 1}\lambda  \Bigg] u^{\alpha }|\nabla u|^2 e^{-\gamma |x|^2}.
	\end{align}
	Firstly, we deal with the term $I_1 + I_2 + I_3$ and we hope it is positive.
	Define
	\begin{equation}
		\Delta :=  \frac{4n}{n + 2} p (1 - \alpha) - \left( \alpha - 2p\frac{n - 1}{n + 2}  \right)^2 .
	\end{equation}
	Choosing
	
	\begin{numcases}{}
		\alpha = 1 - p \text{ when $1 < p \leq \frac{n + 2}{n}$,}\\
		\alpha = -\frac{2p}{n + 2} \text{ when $\frac{n + 2}{n} < p < p_*$, }	
	\end{numcases}
	then for any $1 < p < p_*$ and $n \geq 2$, we have
	\begin{equation}
		\alpha + p + 1 > 0.
	\end{equation}
	
	\begin{claim}\label{Delta}
		For any $1 < p < p_*$ and $n \geq 2$, we have
		\begin{equation}
			\Delta > 0.
		\end{equation}
	\end{claim}
	So there exists $\delta = \delta(n, p) > 0$, such that
	\begin{align}\label{0408equ3}
		I_1 + I_2 + I_3 \geq \delta u^{\alpha}e^{-\gamma |x|^2} E_{i j}^2 + \delta u^{\alpha - 2}|\nabla u|^4 e^{-\gamma |x|^2}. 
	\end{align}
	Besides, using \eqref{0830equ1}, \eqref{0908equ1} and \eqref{0908equ2}, we know
	\begin{equation}\label{0408equ1}
		\begin{aligned}
			 & \frac{n - 1}{n} u^{\alpha}e^{-\gamma |x|^2}(\Delta u)^2 \\
		=& \left( u^{\alpha}e^{-\gamma |x|^2}\Delta u u_{i} \right)_i - \left( u^{\alpha}e^{-\gamma |x|^2} u_{i j} u_{i} \right)_{j} - \frac{n - 1}{n + 2} \alpha \Bigg[ (u^{\alpha - 1}|\nabla u|^2 e^{-\gamma |x|^2}u_i)_i - (\alpha - 1) u^{\alpha - 2}|\nabla u|^4 e^{-\gamma |x|^2} \\
		&- 2u^{\alpha - 1} e^{-\gamma |x|^2}E_{i j}u_i u_j + 2\gamma u^{\alpha - 1}|\nabla u|^2 e^{-\gamma |x|^2}u_i x_i \Bigg] \\
		&+ 2\gamma \frac{n - 1}{n + 1} \Bigg[ \left( u^{\alpha}e^{-\gamma |x|^2}u_j u_{i} x_i \right)_j - \alpha u^{\alpha - 1}|\nabla u|^2 e^{-\gamma |x|^2} u_{i} x_i + 2\gamma u^{\alpha}e^{-\gamma |x|^2}u_j x_j u_{i} x_i - u^{\alpha}e^{-\gamma |x|^2} E_{ij} x_i u_j\\
		&- u^{\alpha}|\nabla u|^2 e^{-\gamma |x|^2} \Bigg]  + \alpha  u^{\alpha - 1}e^{-\gamma |x|^2} E_{i j}  u_{i}u_{j} - 2\gamma  u^{\alpha}e^{-\gamma |x|^2}E_{i j}  u_{i} x_{j}  + u^{\alpha}e^{-\gamma |x|^2} E_{i j}^2 .
		\end{aligned}
	\end{equation}
	Choosing $0 < \varepsilon << \delta$ and substituting \eqref{0408equ3}, \eqref{0408equ1} into \eqref{0408equ2} , we get
	
	\begin{align}
		 -W_3&\geq  \varepsilon u^{\alpha}e^{-\gamma |x|^2} E_{i j}^2 + \varepsilon u^{\alpha - 2}|\nabla u|^4 e^{-\gamma |x|^2} + \varepsilon u^{\alpha} e^{-\gamma |x|^2}(\Delta u)^2 \notag \\
		&+ \left(  - 2\gamma  +\frac{n - 1}{n + 1}\lambda + O(\varepsilon)\right)  u^{\alpha}e^{-\gamma |x|^2}E_{i j}  u_{i} x_{j}  - C u^{\alpha}|\nabla u|^2 e^{-\gamma|x|^2}|x|^2  \notag\\
		&+ \Bigg[ - \frac{n - 1}{n} \lambda(\alpha + p) + 2\gamma\frac{n - 1}{n + 2} p + \alpha \frac{n - 1}{n + 1}\lambda  + O(\varepsilon) \Bigg]u^{\alpha - 1}|\nabla u|^2 e^{-\gamma|x|^2}x_i u_i \notag\\
		& + \Bigg[  - \frac{n - 1}{n} (p - 1)\beta +\frac{n - 1}{n + 1}\lambda  + O(\varepsilon)\Bigg] u^{\alpha }|\nabla u|^2 e^{-\gamma |x|^2},
	\end{align}
	where 
	\begin{align*}
		W_3 & = W_2 - \frac{n}{n - 1}\varepsilon \Bigg[ \left( u^{\alpha}e^{-\gamma |x|^2}\Delta u u_{i} \right)_i - \left( u^{\alpha}e^{-\gamma |x|^2} u_{i j} u_{i} \right)_{j} - \frac{n - 1}{n + 2} \alpha  (u^{\alpha - 1}|\nabla u|^2 e^{-\gamma |x|^2}u_i)_i \\
		&+ 2\gamma \frac{n - 1}{n + 1} \left( u^{\alpha}e^{-\gamma |x|^2}u_j u_{i} x_i \right)_j \Bigg].
	\end{align*}
	Then for any fixed $\lambda, \beta$ $\gamma$,
	\begin{equation}\label{09080827equ}
		\begin{aligned}
			\Rightarrow &  u^{\alpha}e^{-\gamma |x|^2} E_{i j}^2 + u^{\alpha - 2}|\nabla u|^4 e^{-\gamma |x|^2} +  u^{\alpha} e^{-\gamma |x|^2}(\Delta u)^2 \lesssim u^{\alpha}|\nabla u|^2 e^{-\gamma|x|^2}(|x|^2 + 1) - W_3
		& .
		\end{aligned}
	\end{equation}
	Using the equation \eqref{0910equ1} and the facts that $p > 1$ and $\alpha + 2p > \alpha + p + 1 > 0$, by Young's inequality we konw
	
	\begin{align}
		(u^p - \beta u)^2 &= \left( \Delta u - \frac{1}{2}x\cdot \nabla u\right)^2,\notag\\
		\Rightarrow u^{\alpha + 2p} + u^{\alpha + 2} &\lesssim  u^{\alpha}(\Delta u)^2 +  u^{\alpha} |\nabla u|^2 |x|^2 +  u^{\alpha + p + 1} \notag \\
			&\lesssim  u^{\alpha} (\Delta u)^2 +   u^{\alpha} |\nabla u|^2 |x|^2 + \varepsilon u^{\alpha + 2p} + 1, \notag\\
		\label{09080826equ}	\Rightarrow u^{\alpha + 2p} + u^{\alpha + 2} &\lesssim  u^{\alpha} (\Delta u)^2 +  u^{\alpha} |\nabla u|^2 |x|^2 + 1.
	\end{align}
	Substitute \eqref{09080826equ} into \eqref{09080827equ}:
	\begin{equation}\label{09080828equ}
		\begin{aligned}
			\Rightarrow & u^{\alpha}e^{-\gamma |x|^2} E_{i j}^2 +   u^{\alpha - 2}|\nabla u|^4 e^{-\gamma |x|^2} +  u^{\alpha} e^{-\gamma |x|^2}(\Delta u)^2 +  u^{\alpha + 2p} e^{-\gamma |x|^2}  +  u^{\alpha + 2} e^{-\gamma |x|^2} \\
		&\lesssim - W_3 + u^{\alpha}|\nabla u|^2 e^{-\gamma|x|^2}(|x|^2 + 1) + e^{-\gamma |x|^2} .
		\end{aligned}
	\end{equation}
	Let $R > 1$ and $\eta$ be a cut-off function, satisfying 
	\begin{numcases}{}
		\eta \equiv 1, \text{ in $B_R$ ,}\\
		\eta \equiv 0, \text{ in $\mathbb R^n \backslash B_{2R}$ .}	
	\end{numcases}
	Choosing $\theta > 0$ large enough, multiply \eqref{09080828equ} by $\eta^{\theta}$ and integrate by parts:
	\begin{equation}
		\begin{aligned}
			& \int  u^{\alpha}e^{-\gamma |x|^2} |E_{i j}|^2 \eta^{\theta} +   u^{\alpha - 2}|\nabla u|^4 e^{-\gamma |x|^2} \eta^{\theta} +  u^{\alpha} e^{-\gamma |x|^2}(\Delta u)^2 \eta^{\theta} + u^{\alpha + 2p} e^{-\gamma |x|^2}\eta^{\theta} + u^{\alpha + 2} e^{-\gamma |x|^2}\eta^{\theta} \\
			\lesssim & \int  u^{\alpha}|\nabla u|^2 e^{-\gamma|x|^2}(|x|^2 + 1) \eta^{\theta} +  e^{-\gamma |x|^2}\eta^{\theta} +  u^{\alpha}e^{-\gamma |x|^2}\Delta u u_{i} \left(\eta^{\theta} \right)_i +  u^{\alpha}e^{-\gamma |x|^2} u_{i j} u_{i} \left( \eta^{\theta}\right)_{j}   \\
		&+  u^{\alpha - 1}|\nabla u|^2 e^{-\gamma |x|^2}u_i \left( \eta^{\theta} \right)_i+  u^{\alpha}e^{-\gamma |x|^2}u_j u_{i} x_i  \left( \eta^{\theta} \right)_j .\\
		\end{aligned}
	\end{equation}
	Using Young's inequality, we finally get
	\begin{equation}
		\begin{aligned}
			&\int  u^{\alpha - 2}|\nabla u|^4 e^{-\gamma |x|^2} \eta^{\theta}  +  u^{\alpha + 2p} e^{-\gamma |x|^2}\eta^{\theta}  +  u^{\alpha} e^{-\gamma |x|^2}(\Delta u)^2 \eta^{\theta}  + u^{\alpha + 2} e^{-\gamma |x|^2}\eta^{\theta}  \\
			\lesssim & \int  u^{\alpha}|\nabla u|^2 e^{-\gamma|x|^2}(|x|^2 + 1) \eta^{\theta} +  e^{-\gamma |x|^2}\eta^{\theta} + R^{-2} u^{\alpha}|\nabla u|^2 e^{-\gamma |x|^2} \eta^{\theta - 2} + R^{-1} u^{\alpha} |\nabla u|^2 e^{-\gamma |x|^2}|x|\eta^{\theta - 1} \\
			\lesssim & \int   e^{-\gamma |x|^2}\eta^{\theta} +  u^{\alpha}|\nabla u|^2 \Big[ e^{-\gamma|x|^2}(|x|^2 + 1) \eta^{\theta}  + R^{-2}  e^{-\gamma |x|^2} \eta^{\theta - 2} + R^{-1}  e^{-\gamma |x|^2}|x|\eta^{\theta - 1}\Big] \\
			\lesssim & \int   e^{-\gamma |x|^2}\eta^{\theta} +  e^{-\gamma|x|^2} u^{\alpha + 2} \Big[ (|x|^4 + 1) \eta^{\theta}  + R^{-4}   \eta^{\theta - 4} \Big] + \varepsilon u^{\alpha - 2}|\nabla u|^4 e^{-\gamma |x|^2} \eta^{\theta} . \\
		\end{aligned}
	\end{equation}
	Thus it follows that
	\begin{align}\label{0309}
		&\int_{\mathbb R^n}  u^{\alpha - 2}|\nabla u|^4 e^{-\gamma |x|^2} \eta^{\theta}  +  u^{\alpha + 2p} e^{-\gamma |x|^2}\eta^{\theta}  +  u^{\alpha} e^{-\gamma |x|^2}(\Delta u)^2 \eta^{\theta}  + u^{\alpha + 2} e^{-\gamma |x|^2}\eta^{\theta}  \notag\\
			\lesssim & \int_{\mathbb R^n}   e^{-\gamma |x|^2}\eta^{\theta} +  e^{-\gamma|x|^2} u^{\alpha + 2} \Big[ (|x|^4 + 1) \eta^{\theta}  + R^{-4}   \eta^{\theta - 4} \Big] . 
	\end{align}
	There are two different cases depending on $\alpha$:

	\begin{itemize}
		\item If $\alpha + 2 \geq 0$, then by Young's inequality, we get
\begin{equation}
	u^{\alpha + 2} \leq \varepsilon u^{\alpha + 2p} + M_1 .
\end{equation}
Finally we get for any $\beta, \lambda \in \mathbb R$ and $\gamma > 0$
	\begin{equation}\label{0908est}
		\begin{aligned}
			\int_{\mathbb R^n} \left( u^{\alpha - 2}|\nabla u|^4 e^{-\gamma |x|^2}  +  u^{\alpha + 2p} e^{-\gamma |x|^2}  +  u^{\alpha} e^{-\gamma |x|^2}(\Delta u)^2   + u^{\alpha + 2} e^{-\gamma |x|^2} \right)\eta^\theta \lesssim 1.
		\end{aligned}
	\end{equation}
	
	\item If $\alpha + 2 < 0$, for $W > 0$ we consider the function 
	\begin{equation}
		w := W|x|^{-\frac{\beta + 1}{\lambda}}, 
	\end{equation} 
	and the linear operator
	\begin{equation}
		\mathcal L v:= \Delta v - \lambda x\cdot \nabla v - \beta v.
	\end{equation}
	Then we find that 
	\begin{equation}
		\mathcal L u \leq 0
	\end{equation}
	\begin{equation}
		\mathcal L w = W \left[ |x|^{-\frac{\beta + 1}{\lambda}} + \frac{\beta + 1}{\lambda}\cdot\left( \frac{\beta + 1}{\lambda} + 2 - n \right) |x|^{-\frac{\beta + 1}{\lambda} - 2}\right].
	\end{equation}
	So there exists $r_0 > 0$ such that when $|x| > r_0$, it holds that
	\begin{equation}
		\mathcal L w \geq 0.
	\end{equation}
	Letting 
	\begin{equation}
		W = r_0^{\frac{\beta + 1}{\lambda}} \cdot \min_{|x| = r_0} u,
	\end{equation}
	then we have $u \geq w$ when $|x| \geq r_0$. Thus \eqref{0309} becomes that
	\begin{align}\label{0309equ1}
		&\int_{\mathbb R^n}  u^{\alpha - 2}|\nabla u|^4 e^{-\gamma |x|^2} \eta^{\theta}  +  u^{\alpha + 2p} e^{-\gamma |x|^2}\eta^{\theta}  +  u^{\alpha} e^{-\gamma |x|^2}(\Delta u)^2 \eta^{\theta}  + u^{\alpha + 2} e^{-\gamma |x|^2}\eta^{\theta}  \notag\\
		\lesssim & 1 + \int_{\mathbb R^n}  e^{-\gamma|x|^2} u^{\alpha + 2} \Big[ (|x|^4 + 1) \eta^{\theta}  + R^{-4}   \eta^{\theta - 4} \Big] \notag\\
		\lesssim & 1 + \int_{|x| < r_0}  e^{-\gamma|x|^2} u^{\alpha + 2} \Big[ (|x|^4 + 1) \eta^{\theta}  + R^{-4}   \eta^{\theta - 4} \Big] + \int_{|x| \geq r_0}  e^{-\gamma|x|^2} u^{\alpha + 2} \Big[ (|x|^4 + 1) \eta^{\theta}  + R^{-4}   \eta^{\theta - 4} \Big] \notag\\
		\lesssim & 1 + \int_{|x| \geq r_0} |x|^{-\frac{\beta + 1}{\lambda}(\alpha + 2)} e^{-\gamma|x|^2}  \Big[ (|x|^4 + 1) \eta^{\theta}  + R^{-4}   \eta^{\theta - 4} \Big] < \infty.
	\end{align}
	\end{itemize}
	Anyway, combining these two cases and letting $R\rightarrow \infty$, we have already got that
	\begin{align}
		\int_{\mathbb R^n}  u^{\alpha - 2}|\nabla u|^4 e^{-\gamma |x|^2}   +  u^{\alpha + 2p} e^{-\gamma |x|^2}  +  u^{\alpha} e^{-\gamma |x|^2}(\Delta u)^2   + u^{\alpha + 2} e^{-\gamma |x|^2}  < \infty.
	\end{align}
	Since $\alpha + p - 1 \geq 0$, by Young's inequality we get

	\begin{equation}\label{0918equ1}
		u^{p + 1} e^{-\gamma |x|^2} \leq u^{\alpha + 2p} e^{-\gamma |x|^2} + M_2  e^{-\gamma |x|^2},
	\end{equation}
	with $(p_2, q_2) = \left(\frac{\alpha + 2p}{p + 1}, \frac{\alpha + 2p}{\alpha + p - 1}\right)$, 
	and 
	\begin{equation}\label{0918equ2}
		|\nabla u|^2 e^{-\gamma |x|^2} \leq  u^{\alpha + 2p} e^{-\gamma |x|^2} +  u^{\alpha - 2}|\nabla u|^4 e^{-\gamma |x|^2} + M_3 u^{\alpha + 2} e^{-\gamma |x|^2} ,
	\end{equation}
	with $(p_3, q_3, \sigma_3) = \left(-\frac{2p - 2}{\alpha}, 2 , \frac{2p - 2}{\alpha + p - 1}  \right)$. Combining \eqref{0908est}, \eqref{0918equ1} and \eqref{0918equ2}, we get 
	\begin{equation}
			\int_{\mathbb R^n}  u^{p + 1} e^{-\gamma |x|^2} + |\nabla u|^2 e^{-\gamma|x|^2} < \infty .
		\end{equation}

	\end{proof}

	Now we give the proof of Claim \ref{Delta} as follows.
	\begin{proof}[Proof of Claim \ref{Delta}]
		
	\begin{itemize}
		\item When  $1 < p \leq \frac{n + 2}{n}$, we know $\alpha = 1 - p$ and
		\begin{align*}
		\Delta &= \frac{4n}{n + 2} p^2 - \left( 1  - p\frac{3n}{n + 2}  \right)^2\\
		&= - \frac{5n^2 - 8n}{(n + 2)^2}p^2 + \frac{6n}{n + 2}p - 1.
	\end{align*}
	When $p = 1$,
	\begin{equation}
		\Delta = - \frac{5n^2 - 8n}{(n + 2)^2} + \frac{6n}{n + 2} - 1 = \frac{16n - 4}{(n + 2)^2} > 0.
	\end{equation}
	When $p = \frac{n + 2}{n}$,
	\begin{equation}
		\Delta =  \frac{8}{n} > 0.
	\end{equation}
	So we get that when $1 < p \leq \frac{n + 2}{n}$, it holds that $\Delta > 0$.

		\item When $\frac{n + 2}{n} < p < p_*$, we know $\alpha = -\frac{2p}{n + 2}$ and
		\begin{equation}
			\Delta = \frac{4n p}{n + 2}\left( 1 - \frac{n - 2}{n + 2}p\right) > 0.
		\end{equation}

	\end{itemize}
	
	\end{proof}
	Now we can follow the idea of Giga and Kohn in \cite{MR784476} to prove Theorem \ref{0928thm}.
	Define
	\begin{equation}
		\rho := e^{-\frac{1}{4}|x|^2},
	\end{equation}
	then we can rewrite \eqref{0910equ1} as
	\begin{equation}\label{0919equ1}
		\nabla \cdot (\rho \nabla u) - \beta \rho u + \rho u^p = 0.
	\end{equation}
	Using Lemma \ref{Thm1}, we can get the following several results which are important in our proof of Theorem  \ref{0928thm}.
	\begin{lemma}\label{0927lemma1}
		When $n \geq 2$, $\lambda = \frac{1}{2}, \beta = \frac{1}{p - 1}$ and $u \geq 0$ is a solution of \eqref{0910equ1} , we have
		\begin{equation}\label{0928equ1}
			\int  \rho |\nabla u|^2 + \beta \rho u^2 - \rho u^{p + 1} = 0
		\end{equation}
	\end{lemma}
	
	\begin{proof}[Proof of Lemma \ref{0927lemma1}]
		Multiply \eqref{0919equ1} by $-u \eta^2$ and use integration by parts:
		\begin{equation}
			\begin{aligned}
				\int_{\mathbb R^n}  \rho |\nabla u|^2 \eta^2 + \beta \rho u^2 \eta^2 - \rho u^{p + 1} \eta^2 = - \int_{\mathbb R^n} 2 \rho u \eta \nabla u \cdot \nabla \eta.
			\end{aligned}
		\end{equation}
		Besides, by Lemma \ref{Thm1},
		\begin{equation}
			\left\vert \int_{\mathbb R^n} 2 \rho u \eta \nabla u \cdot \nabla \eta \right\vert \leq \int_{\mathbb R^n} \frac{C}{R} \rho u^2 + \frac{C}{R}\rho |\nabla u|^2 \leq \frac{ CM}{R}.
		\end{equation}
		So using dominated convergence theorem, we get when $R$ tends to $\infty$, the conclusion holds.
	\end{proof}

	\begin{lemma}\label{0927lemma2}
		When $n \geq 2$, $\lambda = \frac{1}{2}, \beta = \frac{1}{p - 1}$ and $u \geq 0$ is a solution of \eqref{0910equ1}, we have
		\begin{equation}\label{0928equ2}
			\int_{\mathbb R^n} |x|^2 |\nabla u|^2 \rho + \left[\left(\beta + \frac{1}{2}\right)|x|^2 - n\right]u^2 \rho - |x|^2 u^{p + 1}\rho = 0.
		\end{equation}
	\end{lemma}
	\begin{proof}[Proof of Lemma \ref{0927lemma2}]
		Multiply \eqref{0919equ1} by $-|x|^2 u \eta^2$ and use integration by parts:
		\begin{equation}
			\begin{aligned}
				&-\int |x|^2u \nabla \cdot (\rho \nabla u) \eta^2\\
				 =& \int |x|^2 |\nabla u|^2 \rho \eta^2 + (x\cdot \nabla u^2)\rho \eta^2 + 2|x|^2 u \rho \eta \nabla u \cdot \nabla \eta\\
				=& \int |x|^2 |\nabla u|^2 \rho \eta^2 - n u^2 \rho \eta^2 + \frac{1}{2}|x|^2 u^2 \rho \eta^2 - 2u^2 \rho \eta \nabla \eta \cdot x  + 2|x|^2 u \rho \eta \nabla u \cdot \nabla \eta.
			\end{aligned}
		\end{equation}
		Letting $R$ tends to $\infty$, the conclusion holds.
	\end{proof}

	\begin{lemma}\label{0927lemma3}
		When $n \geq 2$, $\lambda = \frac{1}{2}, \beta = \frac{1}{p - 1}$ and $u \geq 0$ is a solution of \eqref{0910equ1}, we have
		\begin{equation}\label{0928equ3}
			\int \left( \frac{|x|^2}{4} - \frac{n - 2}{2}\right)|\nabla u|^2 \rho + \left( \frac{\beta|x|^2}{4} - \frac{n \beta}{2}\right)u^2 \rho - \left(\frac{|x|^2}{2 p + 2} - \frac{n}{p + 1}\right)u^{p + 1}\rho = 0.
		\end{equation}
	\end{lemma}
	\begin{proof}[Proof of Lemma \ref{0927lemma3}]
		Multiply \eqref{0919equ1} by $-(x \cdot \nabla u) \eta^2$ and use integration by parts:
		\begin{equation}
			\begin{aligned}
				\int (x \cdot \nabla u)\rho \eta^2(\beta u - u^p) &= \int \rho \eta^2 x \cdot \nabla\left( \frac{\beta u^2}{2} - \frac{u^{p + 1}}{p + 1}\right)\\
				&= \int \left( \frac{|x|^2}{2}\eta^2 - n \eta^2 - 2 \eta \nabla \eta \cdot x \right)\rho \left( \frac{\beta u^2}{2} - \frac{u^{p + 1}}{p + 1}\right).\\
			\end{aligned}
		\end{equation}
		and 
		\begin{equation}
			\begin{aligned}
				-\int \eta^2 (x \cdot \nabla u)\nabla \cdot (\rho \nabla u) &= \int \eta^2 (\rho \nabla u)\cdot \nabla( x \cdot \nabla u) + 2\rho\eta (\nabla \eta \cdot \nabla u) \cdot ( x \cdot \nabla u)   \\
				&= \int \rho |\nabla u|^2 \eta^2 + \frac{1}{2}\rho \eta^2  x \cdot \nabla(|\nabla u|^2) + 2\rho\eta (\nabla \eta \cdot \nabla u) \cdot ( x \cdot \nabla u)   \\
				&= \int \rho |\nabla u|^2 \eta^2 - \frac{1}{2}|\nabla u|^2\nabla\cdot( \rho \eta^2  x) + 2\rho\eta (\nabla \eta \cdot \nabla u) \cdot ( x \cdot \nabla u)  . \\
			\end{aligned}
		\end{equation}
		Letting $R$ tends to $\infty$, the conclusion holds.
	\end{proof}
	Now we can give the proof of Theorem \ref{0928thm}:
	\begin{proof}[Proof of Theorem \ref{0928thm}]
		If $\frac{\partial u}{\partial x_2}\equiv 0$, then $n = 1$ is a special case for $n = 2$. So we only need to consider the case $n \geq 2$.
		By taking the linear combination of $\frac{n}{p + 1} \cdot \eqref{0928equ1} - \frac{1}{2(p + 1)} \cdot \eqref{0928equ2} + \eqref{0928equ3}$, we get
		\begin{equation}
			\frac{n + 2 - (n - 2)p}{2p + 2}  \int_{\mathbb R^n} |\nabla u|^2 \rho + \frac{p - 1}{4(p + 1)} \int_{\mathbb R^n} |x|^2 |\nabla u|^2 \rho = 0.
		\end{equation}
		So for any $1 < p < p_*$, we get $|\nabla u| \equiv 0$.
	\end{proof}

	\textbf{Data Availability }Data sharing not applicable to this article as no datasets were generated or analysed during the current study.

	\section*{Declarations}
	
	\textbf{Conflict of interest} The corresponding author states that there is no conflict of
interest. The author has no relevant financial or non-financial interests to disclose.
	
\bibliographystyle{plain} 
\bibliography{possibility.bib}

\end{document}